\documentclass[10pt, reqno]{amsart}
\usepackage{latexsym,amsmath,amssymb,amscd}
\def\today{\ifcase \month \or
   January \or February \or March \or April \or
   May \or June \or July \or August \or
   September \or October \or November \or December \fi
   \space\number\day , \number\year}

\begin{document}

\makeatletter
\@addtoreset{figure}{section}
\def\thefigure{\thesection.\@arabic\c@figure}
\def\fps@figure{h,t}
\@addtoreset{table}{bsection}

\def\thetable{\thesection.\@arabic\c@table}
\def\fps@table{h, t}
\@addtoreset{equation}{section}
\def\theequation{
\arabic{equation}}
\makeatother

\newcommand{\bfi}{\bfseries\itshape}
\newcommand{\3}{\ss{}}

\def\theoremname{Theorem}
\def\propositionname{Proposition}
\def\corollaryname{Corollary}
\def\lemmaname{Lemma}
\def\remarkname{Remark}
\def\stepname{Step}
\def\definitionname{Definition}
\def\exercisename{Exercise}
\def\examplename{Example}
\def\examplesname{Examples}
\def\problemname{Problem}
\def\problemsname{Problems}
\def\proofname{Proof}

\def\@thmcounter#1{\noexpand\arabic{#1}}
\def\@thmcountersep{}
\def\@begintheorem#1#2{\it \trivlist \item[\hskip 
\labelsep{\bf #1\ #2.\quad}]}
\def\@opargbegintheorem#1#2#3{\it \trivlist
      \item[\hskip \labelsep{\bf #1\ #2.\quad{\rm #3}}]}
\makeatother
\newtheorem{theo}{\theoremname}[section]
\newtheorem{propo}[theo]{\propositionname}
\newtheorem{coro}[theo]{\corollaryname}
\newtheorem{lemm}[theo]{\lemmaname}

\newenvironment{theorem}{\begin{theo}\it}{\end{theo}}
\newenvironment{proposition}{\begin{propo}\it}{\end{propo}}
\newenvironment{corollary}{\begin{coro}\it}{\end{coro}}
\newenvironment{lemma}{\begin{lemm}\it}{\end{lemm}}

\newtheorem{rem}[theo]{\remarkname}
\newenvironment{remark}{\begin{rem}\rm}{\end{rem}}

\newtheorem{defin}[theo]{\definitionname} 
\newenvironment{definition}{\begin{defin}\rm}{\end{defin}}

\newtheorem{notat}[theo]{Notation} 
\newenvironment{notation}{\begin{notat}\rm}{\end{notat}}

\newtheorem{ex}[theo]{\examplename}
\newenvironment{example}{\begin{ex}\rm}{\end{ex}}

\newtheorem{exs}[theo]{\examplesname}
\newenvironment{examples}{\begin{exs}\rm}{\end{exs}}

\newtheorem{conj}[theo]{\conjecturename}
\newenvironment{conjecture}{\begin{conj}\rm}{\end{conj}}

\newtheorem{pr}[theo]{\problemname}
\newenvironment{problem}{\begin{pr}\rm}{\end{pr}}

\newcommand{\todo}[1]{\vspace{5 mm}\par \noindent
\framebox{\begin{minipage}[c]{0.85 \textwidth}
\tt #1 \end{minipage}}\vspace{5 mm}\par}

\newcommand{\1}{{\bf 1}}

\newcommand{\hotimes}{\widehat\otimes}

\newcommand{\Ci}{{\mathcal C}^\infty}
\newcommand{\comp}{\circ}
\newcommand{\D}{\text{\bf D}}
\newcommand{\de}{{\rm d}}
\newcommand{\ev}{{\rm ev}}
\newcommand{\id}{{\rm id}}
\newcommand{\ie}{{\rm i}}
\newcommand{\End}{{\rm End}\,}
\newcommand{\Fl}{{\rm Fl}}
\newcommand{\Gr}{{\rm Gr}\,}
\newcommand{\Hom}{{\rm Hom}\,}
\newcommand{\Ker}{{\rm Ker}\,}
\newcommand{\lf}{{\rm l}}
\newcommand{\Ran}{{\rm Ran}\,}
\newcommand{\rad}{{\rm rad}\,}
\newcommand{\spann}{{\rm span}}
\newcommand{\Tr}{{\rm Tr}\,}
\newcommand{\trile}{\trianglelefteq}

\newcommand{\ad}{\mathop{{\rm ad}}\nolimits}
\newcommand{\Exp}{\mathop{{\rm Exp}}\nolimits}
\newcommand{\Log}{\mathop{{\rm Log}}\nolimits}
\newcommand{\GL}{\mathop{{\rm GL}}\nolimits}
\newcommand{\PGL}{\mathop{{\rm PGL}}\nolimits}

\newcommand{\G}{{\rm G}}
\newcommand{\U}{{\rm U}}
\newcommand{\VB}{{\rm VB}}

\newcommand{\cA}{{\mathcal A}}
\newcommand{\cR}{{\mathcal R}}
\newcommand{\Ac}{{\mathcal A}}
\newcommand{\Bc}{{\mathcal B}}
\newcommand{\Cc}{{\mathcal C}}
\newcommand{\Dc}{{\mathcal D}}
\newcommand{\Gc}{{\mathcal G}}
\newcommand{\Hc}{{\mathcal H}}
\newcommand{\Ic}{{\mathcal I}}
\newcommand{\Kc}{{\mathcal K}}
\newcommand{\Lc}{{\mathcal L}}
\newcommand{\Oc}{{\mathcal O}}
\newcommand{\Pc}{{\mathcal P}}
\newcommand{\Sc}{{\mathcal S}}
\newcommand{\Xc}{{\mathcal X}}
\newcommand{\Zc}{{\mathcal Z}}

\newcommand{\Gg}{{\mathfrak G}}
\newcommand{\Hg}{{\mathfrak H}}
\newcommand{\Jg}{{\mathfrak J}}
\newcommand{\Sg}{{\mathfrak S}}

\newcommand{\g}{{\mathfrak g}}
\newcommand{\h}{{\mathfrak h}}
\newcommand{\fj}{{\mathfrak j}}
\newcommand{\s}{{\mathfrak s}}
\newcommand{\fr}{{\mathfrak r}}

\newcommand{\N}{{\mathbb N}}
\newcommand{\Z}{{\mathbb Z}}
\newcommand{\R}{{\mathbb R}\,}
\newcommand{\C}{{\mathbb C}\,}
\newcommand{\K}{{\mathbb K}}
\newcommand{\F}{{\mathbb F}}
\newcommand{\Q}{{\mathbb Q}}
\renewcommand{\H}{{\mathbb H}}

\renewcommand{\:}{\colon}
\newcommand{\0}{{\bf 0}}
\newcommand{\striang}{{\frak{sut}}}
\newcommand{\triang}{{\frak ut}}

\newcommand{\subeq}{\subseteq}
\newcommand{\supeq}{\supseteq}
\newcommand{\ctr}{{}_\rightharpoonup}
\newcommand{\into}{\hookrightarrow}
\newcommand{\eps}{\varepsilon}

\def\onto{\to\mskip-14mu\to} 

\renewcommand{\hat}{\widehat}
\renewcommand{\tilde}{\widetilde}

\newcommand{\Rarrow}{\Rightarrow}
\newcommand{\nin}{\noindent} 
\newcommand{\oline}{\overline}
\newcommand{\Fa}{{\rm Fa}}
\newcommand{\Larrow}{\Leftarrow}
\newcommand{\la}{\langle}
\newcommand{\ra}{\rangle}
\newcommand{\Mon}{{\rm Mon}}
\newcommand{\up}{\mathop{\uparrow}}
\newcommand{\down}{\mathop{\downarrow}}
\newcommand{\res}{\vert}
\newcommand{\sdir}{\times_{sdir}}

\renewcommand{\L}{\mathop{\bf L{}}\nolimits}

\pagestyle{myheadings}
\markboth{\sl CIA}{\sl CIA}


\makeatletter
\title{Geometric characterization of hermitian algebras with 
continuous inversion}
\author{Daniel~Belti\c t\u a} 
\author{Karl-Hermann~Neeb}
\address{Institute of Mathematics ``Simion
Stoilow'' of the Romanian Academy, 
P.O. Box 1-764, RO-014700 Bucharest, Romania}
\email{Daniel.Beltita@imar.ro}
\address{Department of Mathematics, Darmstadt University of Technology, 
Schlossgartenstrasse 7, D-64289 Darmstadt, Germany}
\email{neeb@mathematik.tu-darmstadt.de}
\date{March 10, 2009}

\begin{abstract}
A hermitian algebra is a unital associative ${\mathbb C}$-algebra 
endowed with an involution such that the spectra of self-adjoint elements 
are contained in~${\mathbb R}$. 
In the case of an algebra $\Ac$ endowed 
with a Mackey-complete, locally convex topology such that 
the set of invertible elements is open and the inversion mapping 
is continuous, we construct the smooth structures on 
the appropriate versions of flag manifolds. 
Then we prove that if such a locally convex algebra $\Ac$ is endowed 
with a continuous involution, then it is a hermitian algebra 
if and only if the natural action of all unitary groups 
$U_n(\cA)$ on each flag manifold is transitive. 


\end{abstract}
\makeatother
\maketitle


\section{Introduction}

There exists a deep relationship between 
the locally convex topological algebras with continuous inversion 
and Lie theory. 
This fundamental knowledge originated in the paper \cite{Gl02} 
and has been crucial for many of the subsequent advances 
in the theory of infinite-dimensional Lie groups modeled on locally convex spaces; 
see for instance the survey \cite{Ne06} and the forthcoming monograph~\cite{GN06}. 

On the other hand, a large part of the earlier reseach in this area 
focused on Banach manifolds and their symmetry groups. 
This was naturally related to the spectral theory of Hilbert space operators 
and has lead to many deep results.
A good source of information in this connection is the monograph~\cite{Up85}. 

It is one of the aims of the present paper to relate to each other 
the themes mentioned above, by pointing out that the spectral properties in 
involutive topological algebras are intimately connected with 
differential geometric properties of certain Lie groups and homogeneous spaces 
associated with the algebras under consideration. 
Specifically, we prove that a Mackey-complete algebra $\Ac$ 
with continuous inversion is hermitian 
(i.e., the spectrum of every self-adjoint element is contained in~${\mathbb R}$) 
if and only if the unitary group of the matrix algebra $M_2(\Ac)$ 
acts transitively on the corresponding flag manifolds (Corollary~\ref{char} below). 
This result seems to be new even in the special case when $\Ac$ is a Banach algebra.  
However, the Banach setting is not wide enough to cover 
some of the most important situations, particularly the ones coming from 
the theory of loop groups and gauge theory; see Example~\ref{loop} below for some more details. 

The aforementioned flag manifolds are well known generalizations of the projective spaces 
and were traditionally studied in complex and algebraic geometry. 
There exist a number of infinite-dimensional versions of these compact complex manifolds, 
which are homogeneous Banach manifolds and play an interesting role in operator theory
(see for instance \cite{Up85} or \cite{Bel06}) or in the theory of loop groups 
and certain areas of mathematical physics (\cite{PS90}). 
We recall that the construction of smooth structures on homogeneous spaces of 
infinite-dimensional Lie groups is often a difficult issue 
since it usually relies on the inverse mapping theorem, 
which fails beyond the setting of Banach spaces; 
even in this setting, one additionally needs to find  complements for subspaces in Banach spaces, 
which is often a rather difficult task. 
From this point of view, a by-product of the present research 
turns out particularly important: 
We show that by just using 
an appropriate Gau\ss\ decomposition for matrices with entries 
in a topological algebra, it is possible to construct adequate smooth structures 
on the flag manifolds associated with the continuous inverse algebras (Theorem~\ref{manifolds}). 
The core of our method is a very general lemma interesting on its own, 
which we have recorded in the Appendix~\ref{A}. 

The methods used in this paper are close to those used in \cite{BN05} 
in a very general context 
to obtain manifold structures on homogeneous spaces associated to $3$-graded 
Lie algebras, which leads in particular to natural manifold structures 
on the generalized Gra\3mannians, obtained as orbits of projections. 
We plan to use the results of the present paper, such as 
Corollary~\ref{cia-nest2}, to study representations in spaces 
of sections of holomorphic vector over these manifolds 
(cf.\ \cite{MNS09} for some results in the Banach context). 

\section{Definitions and examples}\label{Sect2}

\begin{notation}\label{bounded}
For an arbitrary unital complex associative algebra $\Ac$ we shall 
use the following notation: 
\begin{itemize}
\item[$\bullet$] $\Ac^\times=\{a\in\Ac\mid(\exists a^{-1}\in\Ac)\quad 
aa^{-1}=a^{-1}a=\1\}$;
\item[$\bullet$] the {\it spectrum} of any $a\in\Ac$ is 
$\sigma_{\Ac}(a) = \sigma(a)
=\{\lambda\in{\mathbb C}\mid \lambda\1-a\not\in\Ac^\times\}$.
\end{itemize}
In addition, if $\Ac$ is endowed with an involution 
$a\mapsto a^*$, the we denote 
\begin{itemize}
\item[$\bullet$] the \textit{unitary group} 
$\U({\Ac})=\{u\in\Ac^\times\mid u^*u=\1\}$; 
\item[$\bullet$] the \textit{set of non-negative elements} 
$\Ac_{+}=\{a\in\Ac\mid a=a^*\text{ and }
  \sigma_{\Ac}(a)\subseteq[0,\infty)\}$;
\item[$\bullet$] the \textit{set of positive elements} 
$\Ac_{+}^\times=\Ac^\times\cap\Ac_{+}$. 
\end{itemize}
Moreover, for any complex vector space $\Xc$ we denote by $\Lc(\Xc)$ 
the set of all linear maps from $\Xc$ into itself.
\end{notation}

\begin{definition}\label{herm}
Let $\Ac$ be an associative unital complex algebra endowed with an involution 
$a\mapsto a^*$. 
We say that $\Ac$ is a \emph{hermitian} algebra 
if $\sigma_{\Ac}(a)\subseteq{\mathbb R}$ whenever $a=a^*\in\Ac$. 
\end{definition}

For later use we record the following sufficient condition 
for an involutive algebra to be hermitian. 

\begin{remark}\label{squares}
If $\Ac$ is an associative unital complex algebra endowed with an involution 
and for every $a=a^*\in\Ac$ we have 
$-1\not\in\sigma_{\Ac}(a^2)$, then $\Ac$ is a hermitian algebra.  

In order to prove this assertion, let us first note that for every $a\in\Ac$ we have 
$\sigma_{\Ac}(a^2)\supseteq\{\lambda^2\mid\lambda\in\sigma_{\Ac}(a)\}$.  
In fact, if $\lambda\in{\mathbb C}$ and 
$\lambda^2\in{\mathbb C}\setminus\sigma_{\Ac}(a^2)$ then 
there exists $b\in\Ac$ with $b(\lambda^2\1-a^2)=(\lambda^2\1-a^2)b=\1$. 
Then $\lambda\1-a$ has both a left inverse and a right inverse, 
hence it belongs to $\Ac^\times$, 
and then $\lambda\in{\mathbb C}\setminus\sigma_{\Ac}(a)$. 

Now let $a=a^*\in\Ac$ and assume that $\sigma_{\Ac}(a)\not\subseteq{\mathbb R}$, 
so there exists $\lambda\in\sigma_{\Ac}(a)\setminus{\mathbb R}$. 
Then $\lambda=x+\ie y$ with $x,y\in{\mathbb R}$ and $y\ne0$, 
whence $\ie\in\sigma(\frac{1}{y}(a-x\1))$. 
Thence $-1=\ie^2\in\sigma((\frac{1}{y}(a-x\1))^2)$ by the above remark, 
and this contradicts the assumption on~$\Ac$ since $\frac{1}{y}(a-x\1)$ 
is a self-adjoint element in $\Ac$. 
\end{remark}

\begin{definition} 
A \emph{continuous inverse algebra} (CIA for short) is 
a Hausdorff locally convex unital algebra $\Ac$ whose unit group 
$\Ac^\times$ is 
open and for which the inversion map $\Ac^\times \to \Ac$, $a \mapsto a^{-1}$ 
is continuous. 

If, in addition, $\Ac$ is complete, then the same arguments 
as for Banach algebras lead to a holomorphic functional calculus 
(\cite{Wae67}, \cite{Gl02}). 
Since completeness is in general not inherited 
by quotients (\cite{Koe69}, \S 31.6), it is natural to consider for 
CIAs the weaker condition that they are {\it FC-complete} in the sense 
that they are closed under holomorphic functional calculus 
(see \cite{BN08}). 
This means 
that for $a \in A$, any open neighborhood $U$ of $\sigma(a)$, 
each holomorphic function $f \in \Oc(U)$ and any contour $\Gamma$ 
around $\sigma(a)$ in $U$, the integral 
$$ f(a) := \frac{1}{2\pi i} 
\oint_\Gamma f(\zeta) (a - \zeta \1)^{-1}\, d\zeta, $$
which defines an element of the completion of $\Ac$, actually exists in $\Ac$. 
\end{definition}

\begin{remark}\label{harald}
A discussion of hermitian algebras with continuous inversion 
including various equivalent characterizations 
in the case of Mackey complete algebras 
can be found in Section~7 of the paper \cite{Bi04}. 
\end{remark}

It is well known that the $C^*$-algebras are hermitian algebras. 
Here are two important examples that go beyond the 
traditional setting of operator algebras. 

\begin{example}\label{group} 
(Group algebras.) 
Let $G$ be any finite-dimensional connected nilpotent Lie group. 
Then the unitization of the group algebra $L^1(G)$ 
is always a hermitian Banach algebra; 
see \cite{Po77}. 
See \cite{Le76}, \cite{Ku79}, and \cite{FGL06} for a discussion 
of more general versions of group algebras that 
give rise to hermitian Banach algebras. 
\end{example}

\begin{example}\label{loop} 
(Loop algebras.) 
Let ${\mathbb T}$ denote the 1-dimensional torus 
and $\Ac:=\Ci({\mathbb T},M_n({\mathbb C}))$ 
for some $n\ge1$. 
Then $\Ac$ endowed with the pointwise defined operations has a natural structure of 
a hermitian (non-Banach) algebra with continuous inversion, 
which plays a central role in the theory of loop groups, 
inasmuch as $\Ac^\times$ is precisely the loop group associated with 
the general linear group $\GL_n({\mathbb C})$ 
(see for instance \cite{PS90}). 

More generally, one can prove that if $\Ac$ is a hermitian algebra with continuous inversion 
and $M$ is a compact topological space, then the algebra 
$\Cc(M,\Ac)$  of continuous $\Ac$-valued functions on $M$ 
with the pointwise defined operations is in turn a hermitian algebra with continuous inversion. 
A similar assertion holds for the algebra $\Ci(M,\Ac)$ of smooth $\Ac$-valued functions on $M$ 
provided that $M$ is a compact manifold;  
see \cite{Gl02} and also Examples~VIII.3 in \cite{Ne06}. 
\end{example}

\begin{definition}\label{nests}
For any unital involutive complex algebra $\Ac$ we denote 
$\Pc({\Ac})=\{p\in\Ac\mid p=p^*=p^2\}$. 
For $p,q\in\Ac$ the notation $p\le q$ means that $qp=p$. 
If, in addition, $p\ne q$, then we write $p<q$. 

Now assume that 
$\delta$: $0=p_0<p_1<\cdots< p_n=\1$ is a finite, totally ordered family 
of elements in $\Pc_{\Ac}$. 
We define the mapping of \emph{diagonal truncation} 
$$\Phi_{\delta}\colon\Ac\to\Ac, \quad 
x\mapsto\sum\limits_{k=1}^n(p_k-p_{k-1})x(p_k-p_{k-1})$$ 
and the unital associative subalgebra of $\Ac$, 
$$\Delta(\delta):=
\{x\in\Ac\mid xp_k=p_kxp_k\text{ for }k=0,1,\dots,n\} 
= \{x\in\Ac\mid xp_k\cA \subeq p_k \cA \text{ for }k=0,1,\dots,n\}, $$
which is the stabilizer of the flag 
$(p_0\cA,\ldots, p_n\cA)$ of right ideals. 
Note that the restriction of $\Phi_{\delta}$ to 
the algebra $\Delta(\delta)$ is multiplicative. 
Also, $\Phi_\delta$ is an idempotent mapping and its range 
is a unital $*$-subalgebra of $\Ac$ which 
can be described as 
$$D(\delta):=\Ran(\Phi_\delta)=\{x\in\Ac\mid xp_k=p_kx\text{ for }k=0,1,\dots,n\}.$$
We shall also denote
$N(\delta):=\Delta(\delta)\cap(\Phi_\delta)^{-1}(\1)$, which is a group 
of invertible elements in $\Delta(\delta)$. 
\end{definition}

\begin{example}\label{matrices} 
Let  $\Bc$ be a unital involutive complex algebra and $n\ge1$ arbitrary. 
Then  
the matrix algebra 
$\Ac:=
M_n(\Bc):=M_n({\mathbb C})\otimes\Bc$ 
has a natural structure of unital involutive complex algebra 
and if we define 
$p_1,p_2,\dots,p_n\in\Ac$ by 
$$p_1=\begin{pmatrix} 
\1 & 0  &  0& \dots &0 \\
0 & 0  &  0& \dots  & 0\\
0 & 0  & 0 & \dots  & 0\\
 \vdots &\vdots     & \vdots    &\ddots  & \vdots \\
 0&   0 &  0  &\dots & 0 \\
\end{pmatrix}, 
p_2=\begin{pmatrix} 
\1 & 0  &  0& \dots &0 \\
0 & \1  &  0& \dots  & 0\\
0 & 0  & 0 & \dots  & 0\\
 \vdots &\vdots     & \vdots    &\ddots  & \vdots \\
 0&   0 &  0  &\dots & 0 \\
\end{pmatrix}, 
\dots, 
p_n=\begin{pmatrix} 
\1 & 0  &  0& \dots &0 \\
0 & \1  &  0& \dots  & 0\\
0 & 0  & \1 & \dots  & 0\\
 \vdots &\vdots     & \vdots    &\ddots  & \vdots \\
 0&   0 &  0  &\dots & \1 \\
\end{pmatrix}, $$
then we get a totally ordered family $\delta$: $0=p_0<p_1<\cdots< p_n=\1$. 
The corresponding mapping $\Phi_\delta\colon\Ac\to\Ac$ is defined by 
replacing the off-diagonal entries by zeros, 
$$\begin{pmatrix} 
b_{11} & b_{12}  &  b_{13} & \dots &b_{1n} \\
b_{21} & b_{22}  &  b_{23} & \dots  & b_{2n} \\
b_{31} & b_{32}  & b_{33}  & \dots  & b_{3n}\\
 \vdots &\vdots     & \vdots    &\ddots  & \vdots \\
 b_{n1}& b_{n2} &  0  &\dots & b_{nn} \\
\end{pmatrix}
\longmapsto 
\begin{pmatrix} 
b_{11} & 0  &  0& \dots &0 \\
0 & b_{22}  &  0& \dots  & 0\\
0 & 0  & b_{33} & \dots  & 0\\
 \vdots &\vdots     & \vdots    &\ddots  & \vdots \\
 0&   0 &  0  &\dots & b_{nn} \\
\end{pmatrix}$$
while  
$\Delta(\delta)$ is the algebra of upper triangular matrices in $\Ac=M_n(\Bc)$. \end{example}

\section{Factorizations}

In this section we provide an extension of Proposition~3.1 in \cite{Pit88} 
from the setting of von Neumann algebras to the one of hermitian algebras 
with continuous inversion; see Proposition~\ref{cia-nest1} below. 
This will be a key tool in our geometric characterization of 
hermitian algebras in terms of flag manifolds 
(Theorem~\ref{trans} and Corollary~\ref{char}). 

\begin{lemma}\label{cia-01}
Let $\Ac$ be a unital associative algebra and $p\in\Pc(\Ac)$.  
Consider the unital algebra $\Ac_p:=p\Ac p$ with the unit $p$, 
and define 
$\iota_0\colon\Ac_p\hookrightarrow\Ac$ (the inclusion map) 
and 
$\iota_1\colon\Ac_p\to\Ac$, $x\mapsto x+(\1-p)$, 
which is an inclusion of multiplicative monoids. 
If $p\ne\1$, then for every $x\in\Ac_p$ we have 
$\sigma_{\Ac}(\iota_0(x))=\sigma_{\Ac_p}(x)\cup\{0\}$ 
 and 
$\sigma_{\Ac}(\iota_1(x))=\sigma_{\Ac_p}(x)\cup\{1\}$. 
\end{lemma}

\begin{proof}
The first of these equalities is equivalent to 
${\mathbb C}^\times\setminus\sigma_{\Ac_p}(x)
={\mathbb C}\setminus\sigma_{\Ac}(\iota_0(x))$. 
To prove the inclusion $\subseteq$, 
let $\lambda\in{\mathbb C}^\times\setminus\sigma_{\Ac_p}(x)$ be arbitrary. 
Then there exists $b\in\Ac_p$ such that 
$(\lambda p-x)b=b(\lambda p-x)=p$.  
Since $bp=pb$, it then follows that 
$(\lambda\1-x)(b+\lambda^{-1}(\1-p))=
((\lambda p-x)+\lambda(\1-p))(b+\lambda^{-1}(\1-p))
=p+(\1-p)=\1$, 
and similarly 
$(b+\lambda^{-1}(\1-p))(\lambda\1-x)=\1$. 
Thus $\lambda\1-x$ has the inverse $b+\lambda^{-1}(\1-p)\in\Ac$. 
Since $\iota_0(x)=x$,  
in particular we get 
$\lambda\in{\mathbb C}\setminus\sigma_{\Ac}(\iota_0(x))$. 
Conversely, assume the latter condition. 
Then there exists $c\in\Ac$ such that $(\lambda\1-x)c=c(\lambda\1-x)=\1$. 
As $px=xp$, it follows at once that $pc=cp\in\Ac_p$, 
and then 
$(\lambda p-x)pc=p(\lambda\1-x)c=p$ 
and $pc(\lambda p-x)=c(\lambda\1-x)p=p$. 
Thus $\lambda p-x$ has the inverse $pc\in\Ac_p$, 
and in particular $\lambda\in{\mathbb C}\setminus\sigma_{\Ac_p}(x)$. 
If $\lambda=0$, then $xc=cx=-\1$ hence $x$ is invertible in $\Ac$; 
on the other hand, since $x\in\Ac_p$, 
we have $x(\1-p)=0$, and then the fact that $x$ is invertible implies 
$\1-p=0$, which contradicts our hypothesis. 
Thus $\lambda\ne0$, and then 
$\lambda\in{\mathbb C}^\times\setminus\sigma_{\Ac_p}(x)$ as we were wishing for. 

The proof of the second of the asserted equalities relies on a similar method. 
We actually check that 
$({\mathbb C}\setminus\{1\})\setminus\sigma_{\Ac_p}(x)
={\mathbb C}\setminus\sigma_{\Ac_p}(\iota_1(x))$. 
If $\lambda$ is an arbitrary element in the left-hand side of this equation 
and $b\in\Ac_p$ is the inverse of $\lambda p-x$, 
then $b+(\lambda-1)^{-1}(\1-p)$ turns out to be the inverse 
of $\lambda\1-\iota_1(x)$. 
Conversely, if $\lambda\in{\mathbb C}\setminus\sigma_{\Ac_p}(\iota_1(x))$ 
and $d$ is the inverse of $\lambda\1-\iota_1(x)$ in $\Ac$, 
then $pd=dp\in\Ac_p$ and this element is the inverse 
of $\lambda p-x$ in $\Ac_p$. 
\end{proof}

The first part of the following statement was also noted in Remark~7.1 in \cite{DG01}; 
see also Lemma~1.2 in \cite{Ne08}. 

\begin{proposition}\label{cia-red}
Let $\Ac$ be a unital algebra with continuous inversion 
and $p\in\Pc(\Ac)$. 
Then the unital algebra $\Ac_p:=p\Ac p$ with the unit $p$ 
and endowed with the relative topology is in turn 
an algebra with continuous inversion. 
In addition if $\Ac$ is a hermitian algebra with continuous inversion 
and $p=p^*$, 
then $\Ac_p$ is in turn a hermitian algebra with respect to the involution 
induced from $\Ac$ and 
$$(\Ac_p)^\times=\iota_1^{-1}(\Ac^\times)$$
where $\iota_1\colon\Ac_p\to\Ac$, $x\mapsto x+\1-p$. 
\end{proposition}

\begin{proof} 
It is clear that the continuous map $\iota_1\colon\Ac_p\to\Ac$, 
is multiplicative, satisfies $(\Ac_p)^\times=\iota^{-1}(\Ac^\times)$, 
and intertwines the inversion mappings 
on $(\Ac_p)^\times$ and $\Ac^\times$. 
This shows that $(\Ac_p)^\times$ is an open subset of $\Ac_p$ 
and the inversion mapping is continuous on $(\Ac_p)^\times$. 

Now assume that $\Ac_p$ is a hermitian algebra and let 
$x=x^*\in\Ac_p$ be arbitrary. 
It then follows by Lemma~\ref{cia-01}
that $\sigma_{\Ac_p}(x)\subseteq\sigma_{\Ac}(x)\subseteq{\mathbb R}$, 
which concludes the proof. 
\end{proof}

In the special case when $\Ac$ is a Banach algebra, the conclusion of 
the following proposition is well known; see for instance the  assertion~(4)  
in the paper \cite{Le76}.  
In order to obtain this result in the general situation, 
we shall rely on a purely algebraic result established in \cite{Wi76}. 

\begin{proposition}\label{wichmann}
Let $\Ac$ be a CIA endowed with a continuous involution $a\mapsto a^*$. 
If $\Ac$ is Mackey complete, then for every $n\ge1$ the following assertions are equivalent: 
\begin{enumerate}
\item\label{n=1} the algebra $\Ac$ is hermitian; 
\item\label{n} the matrix algebra $M_n(\Ac)$ is hermitian. 
\end{enumerate}
\end{proposition}

\begin{proof} 
Firstly recall that  the matrix algebra $M_n(\Ac)$ is in turn a CIA 
(see Corollary~1.2 in \cite{Sw77}). 
The implication \eqref{n}$\Rightarrow$\eqref{n=1} follows by applying Proposition~\ref{cia-red} 
for the self-adjoint idempotent element  $p=p_1$ of Example~\ref{matrices}.  

Conversely, assume that $\Ac$ is a hermitian algebra. 
Then, according to  the Shirali-Ford  theorem for 
algebras with continuous inversion,  
$\Ac$ has the property that  every element of the form $a_1^*a_1+\cdots+a_k^*a_k$   
has the spectrum contained in $[0,\infty)$,  for arbitrary 
$k\ge1$ and $a_1,\dots,a_k\in\Ac$ 
(see Proposition~6.8 and Corollary~7.7 in \cite{Bi04}). 
Then the Theorem proved  in \cite{Wi76} shows that 
the matrix algebra $M_n(\Ac)$ has a similar property. 
In particular, for every matrix $a\in M_n(\Ac)$, the spectrum of $a^*a\in M_n(\Ac)$ 
is contained in $[0,\infty)$. 
Then it follows by the above Remark~\ref{squares} 
that $M_n(\Ac)$ is a hermitian algebra. 
\end{proof}

\begin{proposition}\label{cia-red1}
Let $\Ac$ be a Mackey-complete, hermitian CIA,  
$p=p^*\in\Pc(\Ac)$, and $\Ac_p=p\Ac p$. 
Then for every $a\in\Ac^\times$ we have $pa^*ap\in(\Ac_p)^\times$. 
\end{proposition}

\begin{proof} 
It follows by Proposition~\ref{cia-red} that 
$\Ac_p$ is a hermitian CIA. 
On the other hand, $\Ac_p$ is clearly Mackey-complete for $\Ac$ is so. 
Now for $\Bc=\Ac$ or $\Bc=\Ac_p$ 
define the function 
$$\tau_{\Bc}\colon\Bc\to[0,\infty),\quad 
\tau_\Bc(a)=(r_{\Bc}(a^*a))^{1/2}.$$ 
As a consequence of Lemma~\ref{cia-01}, the functions 
$\tau_{\Ac_p}$ and $\tau_{\Ac}$ agree, 
in the sense that $\tau_{\Ac_p}=\tau_{\Ac}|_{\Ac_p}$. 
It then follows by Theorem~7.3((i)$\iff$(iii)) in \cite{Bi04} that 
the enveloping $C^*$-algebra $C^*(\Ac_p)$ 
is a $C^*$-subalgebra of $C^*(\Ac)$ 
and the diagram 
\begin{equation}\label{*}
\begin{CD}
\Ac @>{\eta_{\Ac}}>> C^*(\Ac) \\
@AAA @AAA \\
\Ac_p @>{\eta_{\Ac_p}}>> C^*(\Ac_p) 
\end{CD}
\end{equation}
is commutative, 
where the vertical arrows are inclusion maps and 
$\eta_{\Bc}\colon\Bc\to C^*(\Bc)$ stands for the canonical $*$-morphism 
for $\Bc=\Ac$ or $\Bc=\Ac_p$ 
(see Definition~3.7 in \cite{Bi04}).  

Note that $C^*(\Ac_p)$ may not contain the unit element 
$\1=\eta_{\Ac}(\1)\in C^*(\Ac)$, however it is 
a unital $C^*$-algebra in its own right, with the unit element 
$\eta_{\Ac}(p)$ ($=\eta_{\Ac_p}(p)$). 
We are going to check that actually 
\begin{equation}\label{**}
C^*(\Ac_p)=C^*(\Ac)_{\eta_{\Ac}(p)}
:=\eta_{\Ac}(p)\cdot C^*(\Ac)\cdot \eta_{\Ac}(p).
\end{equation}
In fact since $p=p^*=p^2$ it follows that 
$\eta_{\Ac}(p)=\eta_{\Ac}(p)^*=\eta_{\Ac}(p)^2$, whence 
$$C^*(\Ac)_{\eta_{\Ac}(p)}=\{b\in C^*(\Ac)\mid b\eta_{\Ac}(p)=\eta_{\Ac}(p)b\}.$$
Since $p$ is the unit element of $\Ac_p$, it follows that 
$\eta_{\Ac_p}(\Ac_p)\subseteq C^*(\Ac)_{\eta_{\Ac}(p)}$, 
so $C^*(\Ac_p)\subseteq C^*(\Ac)_{\eta_{\Ac}(p)}$, 
since the range of $\eta_{\Ac_p}$ is dense in $C^*(\Ac_p)$. 
Conversely, let $b\in C^*(\Ac)_{\eta_{\Ac}(p)}$ arbitrary. 
In particular $b\in C^*(\Ac)$ hence there exists a net $\{a_j\}_{j\in J}$ 
such that $b=\lim\limits_{j\in J}\eta_{\Ac}(a_j)$. 
On the other hand, since $b\in C^*(\Ac)_{\eta_{\Ac}(p)}$ 
we get 
$b=\eta_{\Ac}(p)b\eta_{\Ac}(p)
=\eta_{\Ac}(p)\lim\limits_{j\in J}\eta_{\Ac}(a_j)\eta_{\Ac}(p)
=\lim\limits_{j\in J}\eta_{\Ac}(pa_jp)\in C^*(\Ac_p)$, 
and~\eqref{**} is proved. 

We now come back to the proof of the assertion. 
For $a\in\Ac^\times$ it follows by Proposition~7.5 in \cite{Bi04} 
that $\eta_{\Ac}(a)\in C^*(\Ac)^\times$, 
whence 
$\eta_{\Ac}(p)\eta_{\Ac}(a)^*\eta_{\Ac}(a)\eta_{\Ac}(p)\in
(C^*(\Ac)_{\eta_{\Ac}(p)})^\times$.  
(The latter fact follows for instance by considering a faithful representation 
of $C^*(\Ac)$ on some Hilbert space and using the fact that a positive 
operator on a Hilbert space is invertible if and only if it is bounded from 
below by some positive scalar multiple of the identity.) 
Then by \eqref{**} and \eqref{*} we get 
$\eta_{\Ac_p}(pa^*ap)\in C^*(\Ac_p)^\times$, 
and now by Proposition~7.5 in \cite{Bi04} again 
it follows that $pa^*ap\in(\Ac_p)^\times$. 
\end{proof}

\begin{corollary}\label{cia-red2}
Let $\Ac$ be a Mackey-complete, hermitian algebra with continuous inversion.  
Assume that $p=p^*\in\Pc(\Ac)$ and denote $\Ac_p=p\Ac p$. 
Then for every $a\in\Ac^\times$ there exists $b\in(\Ac_p)^\times$ 
such that $pa^*ap=b^*b$. 
In addition, the invertible element $b$ can be chosen such that 
$b=b^*$ and $\sigma_{\Ac}(b)\subseteq(0,\infty)$.
\end{corollary}

\begin{proof}
It follows by Proposition~\ref{cia-red} that $\Ac_p$ is 
a hermitian algebra with continuous inversion, 
and then by Proposition~7.5 in \cite{Bi04} we get 
$\sigma_{\Ac_p}(pa^*ap)=\sigma_{C^*(\Ac_p)}(\eta_{\Ac_p}(pa^*ap))
\subseteq[0,\infty)$. 
On the other hand $pa^*ap\in(\Ac_p)^\times$ by Proposition~\ref{cia-red1}, 
hence actually $\sigma_{\Ac_p}(pa^*ap)\subseteq(0,\infty)$. 

Now Corollary~4.7 in \cite{Bi04} shows that there exists a unique element 
$b=b^*\in\Ac_p$ such that $b^2=pa^*ap$ and 
$\sigma_{\Ac_p}(b)\subseteq[0,\infty)+\ie{\mathbb R}$. 
Since $b^2=pa^*ap$, it follows (for instance by Remark~4.3 in \cite{Bi04}) 
that 
$\{z^2\mid z\in\sigma_{\Ac_p}(b)\}=\sigma_{\Ac_p}(pa^*ap)\subseteq(0,\infty)$. 
Now the property $\sigma_{\Ac_p}(b)\subseteq[0,\infty)+\ie{\mathbb R}$ 
implies that $\sigma_{\Ac_p}(b)\subseteq(0,\infty)$. 
Thus $b=b^*\in(\Ac_p)^\times$ and $pa^*ap=b^2$, as claimed. 
\end{proof}

To obtain the following statement we shall extend the method of proof of 
Proposition~3.1 in \cite{Pit88}. 

\begin{proposition}\label{cia-nest1}
Let $\Ac$ be a Mackey-complete, hermitian algebra with continuous inversion  
and assume that 
$\delta$: $0=p_0<p_1<\cdots< p_n=\1$ is a finite, totally ordered family 
of self-adjoint elements in $\Pc(\Ac)$. 
Then for every $s\in\Ac^\times$ there exist uniquely determined elements 
$d,b\in\Ac^\times$ such that 
\begin{itemize}
\item[$\bullet$] $s^*s=b^*db$, 
\item[$\bullet$] $\Phi_\delta(d)=d=d^*$ and $\sigma(d)\subseteq(0,\infty)$, 
\item[$\bullet$] $\Phi_\delta(b)=\1$ and $b\in\Delta(\delta)^\times$. 
\end{itemize}
\end{proposition}

\begin{proof} 
The case $n=1$ is clear ---just take $b=\1$ and $d=s^*s$. 
Now assume that the conclusion holds for all families of at most $n$ 
self-adjoint idempotents in 
any Mackey-complete, hermitian algebra with continuous inversion.  
Let $\delta$ be as in the statement and denote by $\delta_{n-1}$ 
the family $0=p_0<p_1<\cdots< p_{n-1}$ in $\Pc_{\Ac_{p_{n-1}}}$. 
Then $\Ac_{p_{n-1}}$ is 
a Mackey-complete, hermitian algebra with continuous inversion 
by Proposition~\ref{cia-red}, hence we can use Corollary~\ref{cia-red2} 
to get $y\in(\Ac_{p_{n-1}})^\times$ 
such that $p_{n-1}s^*sp_{n-1}=y^*y$. 
Then the induction hypothesis implies that 
there exist uniquely determined elements 
$b_{n-1},d_{n-1}\in(\Ac_{p_{n-1}})^\times$ 
such that 
\begin{itemize}
\item[$\bullet$] $p_{n-1}s^*sp_{n-1}=b_{n-1}^*d_{n-1}b_{n-1}$, 
\item[$\bullet$] $\Phi_{\delta_{n-1}}(d_{n-1})=d_{n-1}=d_{n-1}^*$ 
and $\sigma_{\Ac_{p_{n-1}}}(d_{n-1})\subseteq(0,\infty)$, 
\item[$\bullet$] $\Phi_{\delta_{n-1}}(b_{n-1})=p_{n-1}$ and 
$b_{n-1},b_{n-1}^{-1}\in\Delta(\delta_{n-1})$. 
\end{itemize}
>From now on we shall denote the elements in $\Ac$ as $2\times 2$ matrices 
according to the decomposition $\1=p_{n-1}+(\1-p_{n-1})$.  
For instance 
$$s^*s=
\begin{pmatrix}
p_{n-1}s^*sp_{n-1} & p_{n-1}s^*s(\1-p_{n-1}) \\
(\1-p_{n-1})s^*sp_{n-1} & (\1-p_{n-1})s^*s(\1-p_{n-1})
\end{pmatrix}. $$
What we have to do is to find the still unknown entries in the matrices 
$$d=\begin{pmatrix}
d_{n-1} & 0 \\ 
      0 & d_n
\end{pmatrix}
\text{ and } 
b=\begin{pmatrix} 
b_{n-1} & t_n \\
0       & \1-p_n
\end{pmatrix}$$
such that 
$s^*s=b^*db$. 
By multiplying the corresponding matrices 
we see that the latter matrix equation is equivalent to 
the relations 
\begin{itemize}
\item[$\bullet$]
$p_{n-1}s^*sp_{n-1}=b_{n-1}^*d_{n-1}b_{n-1}$, 
\item[$\bullet$] $p_{n-1}s^*s(\1-p_{n-1})=b_{n-1}^*d_{n-1}t_n$, 
and 
\item[$\bullet$] $(\1-p_{n-1})s^*s(\1-p_{n-1})
=t_n^*d_{n-1}t_n+(\1-p_{n-1})d_n(\1-p_{n-1})$. 
\end{itemize}
We already know that the first of these equations is satisfied. 
Since $b_{n-1}$ and $d_{n-1}$ are invertible in $\Ac_{p_{n-1}}$ 
we can solve the second equation for $t_n$ to get 
$$t_n=d_{n-1}^{-1}(b_{n-1}^*)^{-1} p_{n-1}s^*s(\1-p_{n-1}).$$ 
Then the third of the above equations can be solved for $d_n$ 
and by using the above formula for $t_n$ we get  
$$
\begin{aligned}
d_n
 &=(\1-p_{n-1})s^*s(\1-p_{n-1})-t_n^*d_{n-1}t_n \\
 &=(\1-p_{n-1})s^*s(\1-p_{n-1})
    -((\1-p_{n-1})s^*s p_{n-1}b_{n-1}^{-1}d_{n-1}^{-1})
     d_{n-1}
    \underbrace{ (d_{n-1}^{-1}(b_{n-1}^*)^{-1}p_{n-1}s^*s(\1-p_{n-1}))}_{\hskip15pt =t_n} \\
 &=(\1-p_{n-1})(s^*s-s^*s p_{n-1}b_{n-1}^{-1}d_{n-1}^{-1}(b_{n-1}^*)^{-1}
     p_{n-1}s^*s)(\1-p_{n-1}).  
\end{aligned}$$
Since $p_{n-1}s^*sp_{n-1}=b_{n-1}^*d_{n-1}b_{n-1}$, 
we get the following formula for $d_n$ 
only in terms of $s^*s$ and $p_{n-1}$: 
$$d_n=(\1-p_{n-1})(s^*s-s^*s p_{n-1}(p_{n-1}s^*sp_{n-1})^{-1}p_{n-1}s^*s)
    (\1-p_{n-1}).$$
The induction step is complete. 
Note that the property $\sigma(d)\subseteq(0,\infty)$ 
follows by the theorem of Shirali-Ford type for Mackey-complete, 
hermitian algebras with continuous inversion (Corollary~7.7 in \cite{Bi04}) 
since $b$ is clearly invertible and $d=(b^*)^{-1}s^*sb^{-1}$. 
The uniqueness assertion is straightforward. 
See the proof of Proposition~3.1 in \cite{Pit88} for some more details 
which carry over in a direct manner to the present setting. 
\end{proof}

\begin{corollary}\label{cia-nest2}
Let $\Ac$ be a Mackey-complete, hermitian algebra with continuous inversion  
and assume that 
$\delta$: $0=p_0<p_1<\cdots< p_n=\1$ is a finite, totally ordered family 
of self-adjoint elements in $\Pc(\Ac)$. 
Then for every $s\in\Ac^\times$ there exist uniquely determined elements 
$u\in\U({\Ac})$, $a\in D(\delta)_{+}^\times$ and 
$b\in\Delta(\delta)^\times$ 
such that $\Phi_{\delta}(b)=\1$ and $s=uab$. 
\end{corollary}

\begin{proof}
It follows by Proposition~\ref{cia-nest1} that $s^*s=b^*db$, 
where $d$ and $b$ are uniquely determined by the conditions 
$d\in D(\delta)_{+}^\times$, $b\in \Delta(\delta)^\times$,  
and $\Phi_{\delta}(b)=\1$. 
Define $a=d^{1/2}\in \Ran(\Phi_\delta)_{+}^\times$ by 
Corollary~4.7 along with Proposition~7.10 in \cite{Bi04}, 
since $D(\delta)$ ($=\Ran(\Phi_{\delta})$) is a closed unital $*$-subalgebra of $\Ac$. 
Then $s^*s=(ab)^*(ab)$, whence $(s(ab)^{-1})^*(s(ab)^{-1})=\1$. 
Thus $u:=s(ab)^{-1}\in\U({\Ac})$ and $s=uab$. 

For the uniqueness assertion assume that $s=u'a'b'$ is another 
decomposition with similar properties. 
Then $b'^*a'^2b'=s^*s=b^*a^2b$, whence $b=b'$ and $a=a'$ 
according to the uniqueness property from 
the above Proposition~\ref{cia-nest1} along with the uniqueness of 
non-negative square roots. 
\end{proof}

\begin{remark}\label{smoothness}
It follows by the explicit construction performed in the proofs of 
Corollary~\ref{cia-nest2} and Proposition~\ref{cia-nest1} 
that there actually exist real analytic mappings 
$u(\cdot),a(\cdot),b(\cdot)\colon\Ac^\times\to\Ac$ 
such that $u(\cdot)$ takes values in $\U(\Ac)$, 
$a(\cdot)$ takes values in $D(\delta)_{+}^\times$, and 
$b(\cdot)$ takes values in $N(\delta)$, 
and 
$$(\forall s\in\Ac^\times)\quad s=u(s)a(s)b(s). $$
Therefore, by using an appropriate real analytic structure on $\U(\Ac)$ 
constructed by means of the Cayley transform 
(see Section~8 in \cite{BN05}), 
it follows that the multiplication mapping 
$$\U(\Ac)\times D(\delta)_{+}^\times \times N(\delta) \to\Ac^\times,\quad 
(u,a,b)\mapsto uab $$
is a real analytic diffeomorphism. 
\end{remark}

\section{Flag manifolds} 

We are now ready to obtain the main results of the present paper: 
the construction of appropriate smooth structures on the flag manifolds 
associated with algebras with continuous inversion (Theorem~\ref{manifolds}) 
and the characterization of the hermitian algebras in terms of 
transitivity of unitary group actions on flag manifolds 
(Theorem~\ref{trans} and Corollary~\ref{char}). 

\begin{definition}\label{flags}
Let $\Ac$ be a complex associative unital algebra. 
For $p,q\in\Pc(\Ac)$ we shall use the notation $p\sim q$ if and only if 
$pq=q$ and $qp=p$, which is easily seen to be equivalent to  
$p\Ac=q\Ac$. This is an equivalence relation and we denote 
the equivalence class of $p\in\Pc(\Ac)$ by $[p]$. 

Clearly, each $p\cA$ is a complemented right ideal of $\cA$, 
and, conversely, if $\cR \subeq \cA$ is a complemented right ideal, 
then any $\cA$-right module projection 
$\cA \to \cR$ is given by a left multiplication with an  idempotent 
$p$, satisfying $\cR = p\cA$. 
Therefore the \emph{Grassmannian} 
$\Gr(\Ac)=\Pc(\Ac)/\sim$ 
of $\Ac$ can be identified with the set of complemented right ideals 
of $\Ac$ (cf.\ \cite[Subsect.~8.6]{BN04}). 

If $p_1,p_2\in\Pc(\Ac)$, recall that we write $p_1\le p_2$ 
whenever $p_2p_1=p_1$, i.e., $p_1\cA \subeq p_2\cA$.  
In addition, let us recall from Lemma~2.2 in \cite{BG08} that 
the natural action  
$\alpha\colon\Ac^\times\times\Pc(\Ac)\to\Pc(\Ac)$, $(g,p)\mapsto \alpha_g(p):=gpg^{-1}$,  
induces an action $\beta(g,\cR) :=  g\cR$ of $\Ac^\times$ on the set 
$\Gr(\Ac)$ of complemented right ideals. 
Clearly, this action preserves the inclusion order on $\Gr(\cA)$ 
and hence on $\Pc(\cA)$. 

For every $n\ge1$, we shall define the set of $n$-flags in a similar manner 
as above. 
For this purpose, firstly denote 
$$\Pc_n(\Ac)=\{(p_1,\dots,p_n)\in\Pc(\Ac)\times\cdots\times\Pc(\Ac)\mid p_1\le \cdots\le p_n\}$$
and define an equivalence relation $\sim$ on $\Pc_n(\Ac)$ by 
$(p_1,\dots,p_n)\sim(q_1,\dots,q_n)$ 
if and only if $[p_j]=[q_j]$ for $j=1,\dots,n$; 
the equivalence class of $(p_1,\dots,p_n)$ will be denoted by $[(p_1,\dots,p_n)]$. 
Now the \emph{set of $n$-flags} is 
$\Fl_{\Ac}(n)=\Pc_n(\Ac)/\sim$.
There exists a natural injective map 
\begin{equation}\label{flags_vs_gr}
\Fl_{\Ac}(n)\hookrightarrow \Gr(\Ac)^n, \quad 
[(p_1, \ldots, p_n)] \mapsto 
(p_1\cA, \ldots, p_n\cA) 
\end{equation}
whose image consists of all $n$-tuples $(R_1,\ldots, R_n)$ of 
complemented right ideals satisfying 
$R_1 \leq \ldots \leq R_n$. This subset is invariant 
under the natural action 
$\beta^{(n)}\colon\Ac^\times\times\Fl_{\Ac}(n)\to\Fl_{\Ac}(n)$ 
by left multiplication. 
For every $X\in\Fl_{\Ac}(n)$, 
the corresponding \emph{flag manifold}  
$\Fl_{\Ac}(X)$ is defined as the $\Ac^\times$-orbit of $X$.  
\end{definition}

\subsection*{Smooth structure on the flag manifolds}



If $\Ac$ is a CIA, the manifold structure on the corresponding Grassmannian 
was pointed out in Theorem~5.3 in \cite{BN05}; see also 
Remark~7.1 in \cite{DG01} for the Banach case. 
In the case of the flag manifolds associated with a CIA, 
to construct the smooth structure 
one can proceed as follows. 

\begin{remark}\label{Gauss} 
Let $\Ac$ be a CIA. 
If 
$\delta$: $0=p_0<p_1<\cdots< p_n=\1$ is a finite, totally ordered family 
of elements in $\Pc(\Ac)$, 
then 
one proceeds by induction on $n$ to show that 
$g\in\Ac^\times$ has the property 
that $p_jgp_j\in(p_j\Ac p_j)^\times$ for $j=1,\dots,n$ 
if and only if 
there exists a (uniquely determined) \emph{Gau\ss\ decomposition}
\begin{equation}\label{gauss}
g=xdy,\text{ where } d\in D(\delta)^\times,\ x\in N(\1-\delta)\text{ and }y\in N(\delta),
\end{equation}
where we denote by $\1-\delta$ the sequence $0=\1-p_n<\1-p_{n-1}<\cdots< \1-p_0=\1$. 
For instance, let us denote the elements of $\Ac$ as as $2\times 2$ matrices 
according to the decomposition $\1=p_{n-1}+(\1-p_{n-1})$ as 
in the proof of Proposition~\ref{cia-nest1}.  
For every $g\in\Ac^\times$ such that $p_{n-1}gp_{n-1}\in(p_{n-1}\Ac p_{n-1})^\times$ 
we have 
$$\begin{aligned}
g
&=
\begin{pmatrix}
p_{n-1}gp_{n-1} & p_{n-1}g(\1-p_{n-1}) \\
(\1-p_{n-1})gp_{n-1} & (\1-p_{n-1})g(\1-p_{n-1})
\end{pmatrix}  \\
&=\begin{pmatrix}
p_{n-1} & 0 \\
x_{n-1} & \1-p_{n-1}
\end{pmatrix} 
\begin{pmatrix} 
p_{n-1}gp_{n-1} & 0 \\
0 & (\1-p_{n-1})g(\1-p_{n-1})
\end{pmatrix}
\begin{pmatrix}
p_{n-1} & y_1 \\
0 & \1-p_{n-1}
\end{pmatrix}
\end{aligned}$$
where, if we denote by $(p_{n-1}gp_{n-1})^{-1}$ the inverse of 
$p_{n-1}gp_{n-1}$ in the algebra $p_{n-1}\Ac p_{n-1}$, then 
$$x_{n-1}=(\1-p_{n-1})gp_{n-1}(p_{n-1} gp_{n-1})^{-1}
\text{ and }y_1=(p_{n-1}gp_{n-1})^{-1}p_{n-1}g(\1-p_{n-1}).$$ 
Similar computations performed in the algebras $p_{n-1}\Ac p_{n-1},\dots,p_2\Ac p_2$ 
eventually lead to the decomposition~\eqref{gauss}  
under the corresponding assumption on the element~$g\in\Ac^\times$.  
Denote
$$\Omega_\delta=\{g\in\Ac^\times\mid p_jgp_j\in(p_j\Ac p_j)^\times\text{ for }j=1,\dots,n\}$$ 
and  define 
$$\sigma\colon \Omega_\delta\to\Ac,\quad g\mapsto x$$ 
by means of the decomposition~\eqref{gauss}. 
It is clear from the construction that $\sigma$ is a real analytic mapping on the neighborhood 
$\Omega_\delta$ of $\1\in\Ac^\times$ 
\end{remark}

\begin{theorem}\label{manifolds}
If $\Ac$ is a CIA and $[\delta]:=[(p_1,\dots,p_n)]\in\Fl_{\Ac}(n)$, 
then the corresponding manifold $\Fl_{\Ac}([\delta])$ 
has a structure of smooth manifold modeled on a locally convex space such that 
the transitive action $\beta^{(n)}\vert_{\Ac^\times\times\Fl_{\Ac}([\delta])}
\colon\Ac^\times\times\Fl_{\Ac}([\delta])\to\Fl_{\Ac}([\delta])$ 
is smooth, and 
the corresponding orbit mapping 
$\Ac^\times\to\Fl_{\Ac}([\delta])$, $g\mapsto\beta^{(n)}(g,[\delta])$
is smooth and open. 
Moreover, the injective map $\Fl_{\Ac}(n)\hookrightarrow \Gr(\Ac)^n$ 
(see~\eqref{flags_vs_gr}) is continuous. 
\end{theorem}

\begin{proof}
By definition, the flag manifold $\Fl_{\Ac}([\delta])$ 
is transitively acted on by the group $\Ac^\times$, 
and the stabilizer of $[\delta] = (p_1 \cA, \ldots, p_n\cA)$ 
is 
$$ B := \{ g \in \cA^\times \mid g,g^{-1} \in \Delta(\delta) \} 
= \Delta(\delta)^\times, $$
which is a closed subgroup of $\Ac^\times$. 
We thus obtain a bijection 
$$\Fl_{\Ac}([\delta])\simeq\Ac^\times/\Delta(\delta)^\times.$$
Moreover, it is clear that $\Delta(\delta)^\times=D(\delta)^\times N(\delta)$
and it easily follows by the way the Gau\ss\ decomposition 
was constructed in Remark~\ref{Gauss} above
that the multiplication mapping 
$N(\1-\delta)\times\Delta(\delta)^\times\to\Omega_\delta$ 
is a homeomorphism. 
Since $\Ac^\times$ is a locally convex Lie group (see \cite{Gl02}), 
it then follows that a natural smooth structure 
on the flag manifold $\Fl_{\Ac}([\delta])$ can be constructed by using 
Lemma~\ref{quot} in the~Appendix~\ref{A}. 

Note that Lemma~\ref{quot} also implies that 
this smooth structure depends only on 
the point $[\delta]\in\Fl_{\Ac}(n)$ and not on the choice 
of $(p_1,\dots,p_n)\in[\delta]$, since for any other 
$\delta'=(p_1',\dots,p_n')\in[\delta]$, the subgroup 
$N(\1 - \delta')$ leads to the same smooth structure. 

Finally, we recall that the set $\Pc(\Ac)$ is endowed with the topology induced from $\Ac$, 
and then the Grassmannian $\Gr(\Ac)=\Pc(\Ac)/\sim$ 
is endowed with the corresponding quotient topology. 
By using the special case $n = 1$ of the construction above, 
one can see that 
the quotient mapping $\Pc(\Ac)\to\Gr(\Ac)$ is open. 
Now it is easy to see that the natural mapping \eqref{flags_vs_gr} 
from $\Fl_\cA([\delta]) \to \Gr(\cA)^n$ 
is continuous with respect to the above described manifold structure on 
$\Fl_{\Ac}([\delta])$ and the product topology on $\Gr(\Ac)^n$. 
\end{proof}

\subsection*{Unitary groups acting on the flag manifolds}

The second assertion in the following lemma was recorded 
with the same idea of proof 
in Lemma~1.3 in \cite{MS97} for $C^*$-algebras. 
We include here the full details of the proof in order to show 
that the corresponding statement is actually purely algebraic. 

\begin{lemma}\label{mlrsms}
If $\Ac$ is a complex associative unital algebra, then the following assertions hold.
\begin{enumerate}
\item\label{mlr}
If $p,q\in\Pc(\Ac)$ and $p\sim q$, 
then $s:=pq+(\1-p)(\1-q)$ satisfies $s\in\Ac^\times$ and 
$sqs^{-1}=p$. 
\item\label{sms}
If $\Ac$ is additionally endowed with an involution 
that makes it into a hermitian algebra, then for every $e\in\Pc(\Ac)$,  
there exists a unique $p\in\Pc(\Ac)$ such that $p=p^*$ and $p\sim e$, 
namely $p=e(\1-(e^*-e))^{-1}$. 
\end{enumerate}
\end{lemma}

\begin{proof}
\eqref{mlr} 
To see that $s$ is invertible, recall that $p\sim q$ means 
$pq=q$ and $qp=p$, which easily yields $(q-p)^2=0$ and $s=\1+(q-p)$, 
and then $s\in\Ac^\times$ with $s^{-1}=\1-(q-p)$. 
The relation $sqs^{-1}=p$ follows immediately from 
$sq = pq = ps$. 

\eqref{sms}
Firstly note that the element $p\in\Ac$ in the statement is well defined 
since $(e^*-e)^*=-(e^*-e)$, hence the hypothesis that 
$\Ac$ is a hermitian algebra ensures that 
the number $1\in{\mathbb C}$ does not belong to the spectrum of 
$e^*-e\in\Ac$. 
Next, in order to prove the existence assertion, 
we have to check that that element $p\in\Ac$ has 
the required properties 
$$ep=p,\ pe=e, \text{ and }p=p^*=p^2.$$  
The first of these equations follows at once since $e^2=e$. 
The latter equality also implies $(e^*)^2=e^*$, whence 
$e^*(\1-(e^*-e))= e^*e=(\1+(e^*-e))e$. 
Then $(\1+(e^*-e))^{-1}e^*=e(\1-(e^*-e))^{-1}$, 
that is, $p^*=p$. 
Moreover, since $p(\1-(e^*-e))=e$, we get $pe-e=pe^*-p$, 
hence $pe-e=(ep^*-p^*)^*=0$, since we have just seen that $p^*=p$ 
and $ep=p$. 
Thus $pe=e$, and then $pe(\1-(e^*-e))^{-1}=e(\1-(e^*-e))^{-1}$, 
that is, $p^2=p$. 
Consequently $p=p^*\in\Pc(\Ac)$ and $p\sim e$. 

For the uniqueness assertion, assume that $q=q^*\in\Pc(\Ac)$ 
and $q\sim e$. 
In particular $eq=q$, whence $q^*e^*=q^*$, so $qe^*=q$. 
Since also $qe=e$, by subtracting these equalities from each other 
we get $q(e^*-e)=q-e$. 
Thence $q(\1-(e^*-e))=e$, and then 
$q=e(\1-(e^*-e))^{-1}=p$. 
This completes the proof. 
\end{proof}

\begin{theorem}\label{trans}
If $\Ac$ is a Mackey-complete, hermitian CIA, then the corresponding flag manifolds
are transitively acted on by the unitary group of~$\Ac$.
\end{theorem}

\begin{proof}
Let $(p_1,\dots,p_n)\in\Pc_n(\Ac)$. 
We shall prove that the unitary group of $\Ac$ acts transitively on the flag manifold 
$\Fl_{\Ac}([(p_1,\dots,p_n)])$. 
According to Lemma~\ref{mlrsms}\eqref{sms}, we may assume that 
$p_j=p_j^*$ for $j=1,\dots,n$. 
Then $\delta$: $0=p_0<p_1<\cdots< p_n<p_{n+1}=\1$ will be a finite, totally ordered family 
of self-adjoint elements in $\Pc_{\Ac}$. 
Now let $g\in\Ac^\times$ be arbitrary. 
We have to prove that there exists a 
$u\in\U(\Ac)$ such that for $j=1,\dots,n$ we have 
$gp_j\Ac=up_j\Ac$. 
For this purpose we may use the element $u\in\U(\Ac)$ provided by 
Corollary~\ref{cia-nest2}, since the factors in the corresponding decomposition $g=uab$ 
satisfy $ab p_j\cA=p_j\cA$, hence 
$gp_j\Ac=up_j\Ac$, and this completes the proof. 
\end{proof}

\begin{remark}\label{trans_BR07}
The above Theorem~\ref{trans} is a wide extension of  Proposition~2.7 in \cite{BR07}, 
whose method of proof  is specific for finite $W^*$-algebras. 
\end{remark}

\begin{corollary}\label{char}
Let $\Ac$ be a Mackey-complete CIA. 
Then $\Ac$ is hermitian if and only if the matrix algebra $M_2(\Ac)$ 
has the property that each of the corresponding flag manifolds is 
transitively acted on by the unitary group $\U_2(\Ac)$. 
\end{corollary}

\begin{proof}
If $\Ac$ is hermitian, then the Mackey-complete CIA $M_2(\Ac)$ is hermitian 
by Proposition~\ref{wichmann}, 
hence Theorem~\ref{trans} shows that the unitary group $\U_2(\Ac)$ of $M_2(\Ac)$ 
acts transitively on every flag manifold of $M_2(\Ac)$. 

Conversely, if that transitivity condition is satisfied, 
then $\Ac$ has to be a hermitian algebra. 
Indeed, if $\cA$ is not hermitian, then Remark~\ref{squares} shows that there exists a 
hermitian element $a  = a^* \in \cA$ with $\ie \in \sigma(a)$, 
so that $a^2 + \1$ is not invertible. 
Now let $\delta$: $0=p_0<p_1<p_2=\1$ in $M_2(\Ac)$, 
where 
$$p_1 = \begin{pmatrix}
\1 & 0 \\ 
0 & 0 \end{pmatrix}$$ 
and note that the matrix 
$$ g =\begin{pmatrix}
\1 & 0 \\ 
a & \1 \end{pmatrix} \in \GL_2(\cA) $$ 
is not contained in the subset 
$U_2(\cA) \Delta(\delta)^\times$ 
because the entry $(g^* g)_{11} = \1 + a^2$ is not invertible. 
It thus follows that the element $\beta^{(2)}(g,[\delta])\in\Fl_{\Ac}([\delta])$ 
is different from $\beta^{(2)}(u,[\delta])$ for every $u\in\U_2(\Ac)$, 
and this contradicts the assumption that $\U_2(\Ac)$ acts transitively 
on every flag manifold of $M_2(\Ac)$. 
\end{proof}

\begin{rem}
\normalfont
We also note that the matrix 
$$ g =\begin{pmatrix}
a + \ie & a \\ 
a & a- \ie \end{pmatrix} \in \GL_2(\cA) $$ 
satisfies 
$g J g^* J = \1$ for 
$$ J := \begin{pmatrix}\1 & \hfill 0 \\ 0 & -\1 \end{pmatrix}, $$
so that 
$$ g \in \U_{1,1}(\cA) := \{ g \in \GL_2(A) \mid g^{-1} = Jg^* J\}. $$
Since $(a + \ie)(a- \ie) = a^2 + \1$ implies that 
the entry $g_{11} = a + i$ is not invertible, it 
follows that 
$g$ is not contained in the open subset 
$$ N(\1 - \delta) \Delta(\delta)^\times.$$ 
If, conversely, $(\cA,*)$ is hermitian, then 
$$g = \begin{pmatrix}
a & b \\ c & d \end{pmatrix} \in \U_{1,1}(\cA) $$ 
implies that $aa^* = \1 + bb^*$ is invertible, so that 
$a \in \cA^\times$, which implies that 
$$ \U_{1,1}(\cA) \subeq  N(\1 - \delta) \Delta(\delta)^\times$$ 
(cf.\ \cite{Bi04}). 
\end{rem}

%



\appendix

\section{A lemma on manifold structures on homogeneous spaces}\label{A}

The following statement is a more precise version ---in the present setting--- 
for Cor.~5.3 in \cite{DG00}.  
It equally applies to $\Cc^\infty$ and to real analytic manifolds 
modelled on locally convex spaces. 

\begin{lemma}\label{quot}
Let $G$ be a locally convex Lie group and 
assume that $B$ is a closed subgroup of $G$ 
for which there exists a subset $N\subseteq G$ 
such that the following conditions are satisfied: 
\begin{enumerate}
\item $\1\in N$; 
\item $N$ is homeomorphic to an open set in some locally convex space, 
and that homeomorphism makes $N$ into a manifold such that 
the inclusion $N\hookrightarrow G$ is a smooth mapping; 
\item there exists an open neighborhood $\Omega$ of $\1\in G$ 
such that the multiplication mapping 
$$N\times B\to \Omega, \quad (n,b)\to nb$$
is a homeomorphism 
and the corresponding projection mapping 
$\sigma\colon \Omega \to N, nb \mapsto n$ is smooth. 
\end{enumerate}
Then the homogeneous space $G/B$ endowed with the quotient topology 
has a structure of manifold (with the same model space as $N$) 
such that the natural projection $\pi\colon G\to G/B$ 
is smooth and open, and the natural action 
\begin{equation}\label{A1}\tag{A1}
G\times G/B\to G/B,\quad (g,nB)\mapsto L_g(nB):=gnB 
\end{equation}
is smoooth. 

Moreover, if $N' \subeq G$ is another subset satisfying {\rm(1)--(3)}, 
then it defines on $G/B$ the same smooth structure. 
\end{lemma}

\begin{proof}
Let us define 
$$
\psi_{\1} :=\pi|_N\colon N\to G/B,\quad\text{and}\quad 
V_{\1} :=\psi_{\1}(N)\subseteq G/B.
$$
Note that $\psi_{\1}$ is injective 
for, if $\psi_{\1}(n_1)=\psi_{\1}(n_2)$, 
then $n_1\in n_2B$, hence $n_1=n_2$ 
because of the hypothesis that the multiplication mapping 
$N\times B\to\Omega$ bijective. 
Next, for $g\in G$, define 
$$\psi_g :=L_g\circ\pi|_N\colon N\to G/B,\quad\text{and}\quad 
V_g :=\psi_g(N)\subseteq G/B.$$
It is clear that $G/N=\bigcup\limits_{g\in G}V_g$ so, 
in order for the family of bijections
$\{\psi_g\colon N\to V_g\}_{g\in G}$ to define a structure of smooth manifold on $G/B$, 
we still have to prove that for every $g\in G$ 
the set $\psi_g^{-1}(V_g\cap V_{\1})$ is open in $N$ 
and the coordinate change 
$\psi_{\1}^{-1}\circ\psi_g
|_{\psi_g^{-1}(V_g\cap V_{\1})}
\colon\psi_g^{-1}(V_g\cap V_{\1})\to
\psi_{\1}^{-1}(V_g\cap V_{\1})$ is smooth. 
In fact, for $x\in\psi_g^{-1}(V_g\cap V_{\1})\subseteq N$ and 
$x'\in\psi_{\1}^{-1}(V_g\cap V_{\1})\subseteq N$ we have 
$$(\psi_{\1}^{-1}\circ\psi_g)(x)=x'
\iff \psi_g(x)=\psi_{\1}(x')
\iff gxB=x'B
\iff gx\in x'B 
\iff gx\in\Omega \text{ and }x'=\sigma(gx)$$
and therefore $\psi_g^{-1}(V_g\cap V_{\1})=(g^{-1}\Omega)\cap N$ is open in $N$ 
and $\psi_{\1}^{-1}\circ\psi_g\colon(g^{-1}\Omega)\cap N\to N$, $x\mapsto\sigma(gx)$ 
is smooth because of the hypothesis on~$N$. 

Thus we get a manifold structure on $G/B$ such that the translation mapping 
$L_g\colon G/B\to G/B$ is smooth for every $g\in G$. 
Since $\psi_{\1}\colon V_{\1}\to N\hookrightarrow G$ is a continuous 
cross-section of $\pi\colon G\to G/B$ over the open subset $V_\1$ 
of the manifold $G/B$, 
it follows that the projection $\pi$ is an open mapping. 
Moreover, since the multiplication mapping $N\times B\to\Omega$ 
is a bijection, we get $\pi(\Omega)=\pi(N)=V_{\1}\subseteq G/B$, 
and then it easily follows that $\pi\colon G\to G/B$ is continuous. 
Since $\pi$ is also open, 
the topology underlying the manifold structure of $G/B$ 
coincides with the quotient topology. 
As $B$ is a closed subgroup of $G$, the corresponding topology of $G/B$ 
is Hausdorff. 

It is clear from the above construction that \eqref{A1} 
is an action of $G$ by smooth transforms on~$G/B$. 
Therefore, for proving that~\eqref{A1} is actually a smooth  mapping, 
it is enough to check that it is smooth on some neighborhood of 
the point $(\1,\1 B)\in G\times G/B$. 
To this end, let $U$ be an open neighborhood of $\1\in G$ 
and $V_{\1}'$ an open neighborhood of $\1 B\in B/B$ such that 
$L_g(x)\in V_{\1}$ for all $g\in U$ and $x\in V_{\1}'$. 
Then we have a commutative diagram 
$$
\begin{CD}
U\times\pi^{-1}(V_{\1}') @>>> \pi^{-1}(V_{\1}) \\
@A{\id_U\times(\psi_{\1})^{-1}\vert_{V_{\1}'}}AA @VV{\pi}V \\
U\times V_{\1}' @>>> V_{\1}
\end{CD}
$$
whose upper horizontal arrow is given by the multiplication in $G$, 
while the lower horizontal arrow is the appropriate restriction 
of the mapping~\eqref{A1}. 
Thus the latter restriction of the group action~\eqref{A1} 
factorizes as a composition of smooth mappings, 
and it is therefore smooth on the neighborhood $U\times V_{\1}'$ 
of the point $(\1,\1 B)\in G\times G/B$. 

Finally, suppose that $N' \subeq G$ is another subset satisfying {\rm(1)-(3)}. 
If $N'$ is an open subset of $N$, we clearly obtain the same manifold 
structure on $G/B$ because $G$ acts by diffeomorphisms, so that it is 
determined by the smooth structure in a neighborhood of~$\pi(\1)$. 
In the general case, the subset 
$\tilde N' := N' \cap \Omega$ is open in $N'$, 
and (3) implies that $N''B$ is open in $G$, so that 
$\tilde N := N \cap \tilde N'B$ is open in $N$. 
Passing from $N$ to $\tilde N$ and from $N'$ to $\tilde N'$, we 
may therefore assume that $NB = N'B$. Now (3) implies that the 
map $N' \to N, n' \mapsto \sigma(n')$ is a diffeomorphism whose 
inverse is given by $N \to N', n \mapsto \sigma'(n)$. 
From that we conclude that the map $\psi_\1' \circ \psi_\1 \: NB \to N'B$ 
is a diffeomorphism, which in turn implies that $N$ and $N'$ define 
the same smooth structure on $G/B$. 
\end{proof}


\end{document}